\documentclass{article}
\usepackage{graphicx}
\usepackage{epstopdf}
\usepackage{amsmath}
\usepackage{amssymb}
\usepackage{bbm}
\usepackage{enumerate}
\usepackage[T1]{fontenc}
\usepackage[latin1]{inputenc}
\usepackage{hyperref}
\usepackage{color}
\usepackage{curves}
\usepackage{tikz}
\usepackage[hang]{caption}
\usepackage{ntheorem}
\usepackage{float}
\usepackage{mathrsfs}

\title{Water transport on finite graphs}
\author{Timo Vilkas
\\\normalsize Lunds universitet}

\theoremstyle{break}
\newtheorem{theorem}{Theorem}[section]

\newtheorem{lemma}{Lemma}[section]
\newtheorem{proposition}{Proposition}[section]
\theorembodyfont{\upshape}
\newtheorem{definition}{Definition}
\newtheorem*{remark}{Remark}
\newtheorem{example}{Example}[section]

\makeatletter
\let\c@proposition\c@theorem
\let\c@lemma\c@theorem
\let\c@corollary\c@theorem
\makeatother

\newenvironment{proof}{\noindent{\sc Proof:}}{\vspace{-0.5cm}~\hfill $\square$\vspace{0.5cm}}
\newenvironment{nproof}[1]{\noindent{\sc Proof #1:}}{\vspace{-1em}~\hfill $\square$\vspace{2em}}

\newcommand\N{\mathbb{N}}

\newcommand\Z{\mathbb{Z}}

\renewcommand\epsilon{\varepsilon}
\renewcommand\phi{\varphi}

\definecolor{darkblue}{rgb}{0,0,.5}
\hypersetup{colorlinks=true, breaklinks=true, linkcolor=blue, 
citecolor=darkblue, menucolor=blue, urlcolor=blue}

\begin{document}
\newpage
\maketitle
%%%%%%%%%%%%%%%%%%%%%%%%%%%
% abstract, keywords and Subject classification are optional.
%%%%%%%%%%%%%%%%%%%%%%%%%%%
\begin{abstract}
    Consider a simple finite graph and its nodes to represent identical water barrels (containing different amounts of water) on a level plane. Each edge corresponds to a (locked, water-filled) pipe connecting two barrels below the plane. We fix one node $v$ and consider the optimization problem relating to the maximum value to which the level in $v$ can be raised without pumps, i.e.\ by opening/closing pipes in a suitable order.
    
    This fairly natural optimization problem originated from the analysis of an opinion formation process and proved to be not only sufficiently intricate in order to be of independent interest, but also difficult from an 
    algorithmic point of view.
\end{abstract}

% Most people don't use these, so they are "commented out"
% by starting the lines with a "%"

%%%%%%%%%%%%%%%%%%%%%%
% % Here is the start of the Text
%%%%%%%%%%%%%%%%%%%%%%
\section{Introduction}

Imagine a finite number of identical rainwater tanks, with a given capacity of $c$ liters, placed on a plane. Some of them are connected by underground (and water-filled) pipes, which can be closed by locks. After a heavy rain, the tanks collected (potentially different) amounts of water and one can consider the optimization problem to raise the water level in a fixed tank by opening and closing the locks.

To cast the problem into an apt mathematical model, consider a finite undirected graph $G=(V,E)$, which we can assume without loss of generality to be simple (i.e.\ having neither loops nor multiple edges). The barrels are represented by the nodes (with an assigned value corresponding to the water level in it) and the pipes are represented by the edges of the graph. We start with all pipes locked, an initial water profile $\big(\eta_0(u)\big)_{u\in V}\in[0,c]^V$ and a given target vertex $v\in V$, in which the water level is to be maximized. To accomplish this we can open a lock, which will lead to the amounts $a$ and $b$ in the incident barrels to level out -- partly if we close the lock early or to their average $\tfrac{a+b}{2}$ if we don't, see Figure \ref{barrels}.

\begin{figure}[ht]
     \centering
     \includegraphics[scale=0.9]{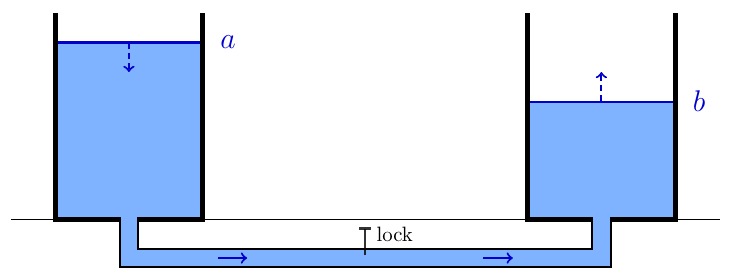}
     \caption{Converging water levels after opening a lock. \label{barrels}}
\end{figure}

In order to formalize this, let us define such a {\em move} to be represented by $(e,\mu)$, where $e=\langle x,y\rangle\in E$ is the pipe we opened and $\mu\in[0,\frac12]$ is the fraction of the difference moved from the fuller into the emptier incident barrel, i.e.\ if we had the profile $\big(\eta_{k}(u)\big)_{u\in V}$ before, it will be $\big(\eta_{k+1}(u)\big)_{u\in V}$ given by
 \begin{equation}\label{update}
    \eta_{k+1}(u)=\begin{cases}
    	\eta_{k}(x)+\mu\,\big(\eta_{k}(y)-\eta_{k}(x)\big),& \text{for } u=x\\
    	\eta_{k}(y)+\mu\,\big(\eta_{k}(x)-\eta_{k}(y)\big),& \text{for } u=y\\
    	\eta_{k}(u),& \text{for } u\in V\setminus\{x,y\}
    \end{cases}
 \end{equation}
after the move $(e,\mu)$. Based on this elementary building block, we can define the objective of the optimization problem:

\begin{definition}\label{kappa}
	Given graph $G=(V,E)$, initial water profile $\{\eta_0(u)\}_{u\in V}$ and a {\em target vertex} $v\in V$,
	we can execute a (finite) {\em move sequence}, represented by an ordered list of edges and fractions,
	$\phi\in (E\times[0,\tfrac12])^t$, where $\phi_k=(e_k,\mu_k)$ corresponds to a single move as in \eqref{update}. By $\kappa(v)$ we denote the supremum over all water levels achievable at $v$ with finite move sequences.
\end{definition}

Obviously, $\kappa(v)$ wouldn't change if we allowed infinite move sequences, as those can be approximated by finite ones. In a preliminary section, we will verify that neither restricting to complete moves (where $\mu=\frac12$)
nor allowing {\em hypermoves}, i.e.\ to open several pipes simultaneously, changes $\kappa(v)$, cf.\ Lemmas \ref{simplif} and \ref{simplif2}.

\subsection{Main results}

Based on these auxiliary results, we are going to prove that there always exists a {\em finite} sequence of hypermoves achieving water level $\kappa(v)$ at the target vertex (Theorem \ref{finitemacro}). In Section \ref{complexity} some light will be shed on heuristic algorithmic approaches to this optimization problem and in Subsection \ref{complex} we show that the water transport problem is NP-hard (Theorem \ref{NPhard}). Besides that, an abundance of example instances will be given both to illustrate the problems with intuitive heuristics and to deal with tractable graphs such as paths and the complete graph, for which the calculation of $\kappa(v)$ can be done explicitely. The last section contains further observations and related open problems.

\subsection{Related work}

First and foremost, this work complements an analysis of the same optimization problem on infinite graphs with i.i.d.\ random initial water levels \cite{infinite}, in which it was established that the two-sided infinite path $\Z$ behaves much more like a finite graph (in the sense that $\kappa(v)$ is random) as opposed to all other infinite, quasi-transitive graphs, for which $\kappa(v)$ is deterministic.

Readers familiar with mathematical models for opinion formation in groups might find that (\ref{update})
in essence resembles update rules in average preserving models for consensus formation in social networks.
Indeed, the problem at hand arises naturally in the theoretical analysis of such models 
related to the question of how extreme an opinion a fixed agent can get, given an initial opinion profile on a specified network graph (e.g.\ for the bounded confidence model introduced by Deffuant et al.\ in 2000, which was analyzed in \cite{ShareDrink}, \cite{Deffuant} and \cite{Lanchier}).

In order to tackle this question, Häggström \cite{ShareDrink} introduced a non-random pairwise averaging
process, which he proposed to call {\em Sharing a drink} (SAD). It was originally considered on $\Z$ only,
but can readily be generalized to any graph (see Definition \ref{SAD}). In fact, SAD is both a special case of
(as the initial profile considered there is $\eta_0=\delta_u$) and dual to the water transport described above, made precise in Lemma \ref{dual}. As Thm.\ 2.3 in \cite{ShareDrink} it was shown that on $\Z$ (and as a consequence also on any finite path) with initial profile $\delta_u$ it holds $\kappa(v)=\frac{1}{d(u,v)+1}$, where $d(u,v)$ is the graph distance between vertices $u$ and $v$. This result was recently generalized to simple graphs in full generality \cite{SAD}.

The duality of SAD to averaging processes in general, turned out to be a useful tool even in the analysis of other consensus formation processes, such as the edge-averaging process (cf.\ \cite{avg2} and \cite{avg}).

\section{Preliminaries}
\subsection{Sharing a drink}\label{sec2}

Let us first repeat the formal definition of the SAD-process:
\begin{definition}\label{SAD}
For a graph $G=(V,E)$ and some fixed vertex $u\in V$, let $\delta_u:\ V\to[0,1]^V$ denote the indicator of vertex $u$, i.e.
$$\delta_u(w)=\begin{cases}1\quad\text{for } w=u\\ 0 \quad\text{for } w\neq u.\end{cases}$$
Given a finite move sequence as above, say $\phi=\big((e_1,\mu_1),(e_2,\mu_2),\dots,(e_t,\mu_t)\big)$, the interaction process called {\em Sharing a drink (SAD)} initiated from vertex $u$ starts with initial water profile $\eta_0=\delta_u$, proceeds along $\phi$ according to \eqref{update} and after these $t$ moves ends up with a {\em terminal} water profile $\eta_t\in[0,1]^V$.
\end{definition}

When implementing a (finite) move sequence on an arbitrary initial water profile $\eta_0$, the terminal values at each vertex will obviously be convex combinations of the initial values. How much each vertex $u$ contributed to the amount of water at a given vertex $v$ is captured by the corresponding {\em dual} SAD-process:

\begin{lemma}[Duality]\label{dual}
Consider an initial water profile $\{\eta_0(u)\}_{u\in V}$ on $G=(V,E)$ and a finite move sequence $\phi$ of length
$t\in\N$. Then it holds for all $v\in V$ that
\begin{equation}\label{convcomb}\eta_t(v)=\sum_{u\in V}\xi_{u,v}(t)\,\eta_0(u),\end{equation}
where $\xi_{u,v}(t)$ is the value at vertex $u$ in the terminal profile of the SAD-process initiated from vertex $v$ with respect to the given move sequence in reversed (time) order, i.e.\ $\overleftarrow{\phi}=\big((e_t,\mu_t),\dots,(e_2,\mu_2),(e_1,\mu_1)\big)$.
\end{lemma}
This is in essence Lemma 3.1 in \cite{ShareDrink}, just adapted to discrete time (see Lemma 2.1 in \cite{infinite}).
Next, let us for convenience also repeat Lemma 2.3 in \cite{infinite}, which is a collection of results about SAD-profiles extracted from $\cite{ShareDrink}$, here:

\begin{lemma}\label{collection}
	Consider the SAD-process on a graph $G$, started from vertex $v$, i.e.\ with initial profile $\delta_v$. Then the following holds:
	\begin{enumerate}[(a)]
		\item If $G$ is a path, all achievable SAD-profiles are unimodal.
		\item If $G$ is a path and $v$ only shares the water to one side, it will remain a mode of the SAD-profile.
		\item The supremum over all achievable SAD-profiles at another vertex $u$ equals $\tfrac{1}{d(u,v)+1}$,
		where $d(u,v)$ is the graph distance between $u$ and $v$.
	\end{enumerate}
\end{lemma}	
In fact, part (a) holds more generally in the sense that unimodality is preserved on paths even for initial profiless other than $\delta_v$. 
Further, it is not hard to see that parts (a) and (b) fail if $G$ is not a path. Likewise, that the value proposed in (c) is within reach is rather obvious (especially with Lemma \ref{evenout} below): just share the drink along the shortest path from $v$ to $u$. The intricate part is to verify that you cannot do better. Indeed, part (c) was originally proved (as Thm.\ 2.3) in $\cite{ShareDrink}$ for paths only, but recently extended to general graphs for $\mu=\frac12$, using a clever combinatorial construction \cite{SAD}. The straight-forward generalization to arbitrary and different $\mu_k$ (based on Lemma \ref{simplif} below) can be found as Lemma 2.4 in \cite{avg}.

Before we turn to the task of raising water levels, let us provide one more auxiliary result, which is Lemma 2.2 in \cite{infinite} and immediately follows from the energy argument that was used in the proof of Thm.\ 2.3 in \cite{ShareDrink}:
\begin{lemma}\label{evenout}
Consider a graph $G=(V,E)$ with initial water profile $\{\eta_0(u)\}_{u\in V}$ and fix a set $A\subseteq V$
together with a collection $E_A\subseteq E$ of edges inside $A$ connecting all vertices in $A$.

Opening the pipes in $E_A$ -- and none connecting $A$ and $V\setminus A$ -- in repetitive sweeps for times long enough
such that $\mu_k\geq\epsilon$ for some fixed $\epsilon>0$ in each round (cf.\ (\ref{update})),
will make the water levels inside $A$ approach their average, i.e.\ $\eta_t(v)$
converges to $\tfrac{1}{|A|} \sum_{u\in A}\eta_0(u)$ for all $v\in A$.
Consequently, for any $u,v\in A$, the corresponding terminal value in $u$ of the dual SAD-process started with $\delta_v$ converges to $\xi_{u,v}=\tfrac{1}{|A|}$.
\end{lemma}

%\begin{proof}
% Let us define the energy after round $k$ inside $A$ by
% $$W_k(A)=\sum_{u\in A}\big(\eta_k(u)\big)^2.$$
% A short calculation reveals that an update of the form (\ref{update}) reduces the energy by
% $2\mu_k^{\,2}\,(b-a)^2$, where the updated water levels were $a$ and $b$ respectively. If $\mu_k$ is bounded away from $0$,
% the fact that $W_k(A)\geq0$ for all $k$ entails that the difference in water levels $|b-a|$ before a pipe is opened
% can be larger than any fixed positive value only finitely many times. In effect, since any pipe in $E_A$ is opened
% repetitively we must have $|\eta_k(u)-\eta_k(v)|\to 0$ as $k\to\infty$ for all edges $\langle u,v\rangle\in E_A$.
% As the updates are average preserving, the first part of the claim follows from the fact that $E_A$ connects $A$.
% 
% The second part of the lemma follows by applying the same argument to the dual SAD-process.
% \end{proof}
% \vspace*{1em}

\subsection{Meta-sequences and hypermoves}\label{macro}

We will first formally define
the concept of optimal move sequences and then carve out some of their properties.

\begin{definition}\label{opt}
Let $\phi\in (E\times[0,\tfrac12])^t$,
where $\phi_k=(e_k,\mu_k)$, be called an {\em optimal} (finite) move sequence if opening the pipes $e_1,\dots,e_t$ in chronological order, each for the period of time that corresponds
to $\mu_k$ in (\ref{update}), will lead to the terminal value $\eta_t(v)=\kappa(v)$.

If no finite optimal move sequence exists, let us call $\Phi=\{\phi^{(m)},\;m\in\N\}$ an {\em optimal meta-sequence of moves}, provided that $\phi^{(m)}\in (E\times[0,\tfrac12])^{t_m}$ is a finite move sequence for each $m\in\N$,
achieving $\eta_{t_m}(v)> \kappa(v)-\tfrac1m$, and the terminal SAD-profiles $\{\xi_{u,v}(t_m)\}_{u\in V}$ dual to $\phi^{(m)}$ converge pointwise to a limit $\{\xi_{u,v}\}_{u\in V}$ as $m\to\infty$.
\end{definition}

Observe that the restriction for the terminal profiles of the dual SAD-processes in an optimal meta-sequence to converge pointwise is more a technical one: Since $G$ is finite, a simple compactness argument (proceeding to subsequences of $\Phi$ with converging terminal dual SAD-profiles) guarantees that such exist. It simply excludes jumping back and forth between finite move sequences approximating (different) optimal limiting sequences.

\begin{lemma}\label{simplif}
	For any (finite) network $G=(V,E)$, target vertex $v$ and initial water profile $\eta_0$, there exists an optimal move \mbox{(meta-)sequence} $\Phi$ and without loss of generality we can assume all involved moves to be complete, i.e.\ $\mu_k=\tfrac12$ in all its moves.
\end{lemma}

\begin{proof}
By the very definition of $\kappa(v)$, the existence of an optimal finite or meta-sequence of moves
is guaranteed: Let $A\subseteq[0,1]^V$ denote the set of achievable terminal profiles for SAD-processes started in $v$ (including finitely many moves).
Its closure $\overline{A}$ in 
$([0,1]^V, \lVert\,.\,\rVert_\infty)$ is bounded and therefore compact. Given the initial water profile
$\{\eta_0(u)\}_{u\in V}$, the function, which maps the profiles to the corresponding
terminal water level at $v$,
$$f:=\begin{cases}[0,1]^V\to[0,C]\\
                  \{\xi_{u,v}\}_{u\in V}\mapsto \sum_{u\in V}\limits\xi_{u,v}\,\eta_0(u)\end{cases}$$
is continuous. Hence there exists a non-empty closed subset $F$ of $\overline{A}$ on which $f$ achieves its
maximum $\kappa(v)$ over $\overline{A}$. The SAD-profiles dual to finite optimal move sequences are given by
$F\cap A$. If $F\cap A=\emptyset$, there exists a sequence $(\xi^{(m)})_{m\in\N}$ of elements in $A$, converging in
maximum norm to an element of $F$. For every profile $\xi^{(m)}$, there exists a finite move sequence
$\phi^{(m)}\in (E\times[0,\tfrac12])^{t_m}$ dual to the SAD-process generating $\xi^{(m)}$ as terminal profile.
These can be gathered to a collection of finite move sequences, $\Phi=\{\phi^{(m)},\;m\in\N\}$.
Without loss of generality, we can assume for all $m\in\N$, that $\phi^{(m)}$ achieves
$\eta_{t_m}(v)=f(\xi^{(m)})> \kappa(v)-\tfrac1m$ (by passing on to a subsequence if necessary).
This turns $\Phi$ into an optimal meta-sequence of moves.
                  
Assume now that the first move in a finite sequence $\phi\in (E\times[0,\tfrac12])^t$ is to open pipe
$e_1=\langle x,y\rangle$ for a time corresponding to $\mu_1\in[0,\tfrac12]$ in (\ref{update}).
Without loss of generality we can assume $\eta_0(x)\geq\eta_0(y)$. Let $\xi_{u,v}(t)$ be dual to $\phi$ (cf.\ Lemma \ref{dual}) and note that it follows from duality (the last move of the dual SAD-process corresponds to the first move of $\phi$):
\[\xi_{u,v}(t)=\begin{cases}
		\xi_{u,v}(t-1),& \text{for } u\in V\setminus\{x,y\}\\
		(1-\mu_1)\,\xi_{x,v}(t-1)+\mu_1\,\xi_{y,v}(t-1),& \text{for } u=x\\
		\mu_1\,\xi_{x,v}(t-1)+(1-\mu_1)\,\xi_{y,v}(t-1),& \text{for } u=y.
	\end{cases}\] 
Hence it holds (using Lemma \ref{dual})
\begin{align*}\eta_t(v)&=\sum_{u\in V}\xi_{u,v}(t)\,\eta_0(u)\\
	&=\mu_1\,\big(\xi_{y,v}(t-1)-\xi_{x,v}(t-1)\big)\,\big(\eta_0(x)-\eta_0(y)\big)+\sum_{u\in V}\xi_{u,v}(t-1)\,\eta_0(u)\end{align*}
and there are two cases to distinguish: either we have $\xi_{x,v}(t-1)\geq\xi_{y,v}(t-1)$ or $\xi_{x,v}(t-1)<\xi_{y,v}(t-1)$.
In the first case changing $\mu_1$ to $0$, i.e.\ erasing the first move will not decrease the water level
finally achieved at $v$. In the second case, the same holds for changing $\mu_1$ to
$\tfrac12$. Since we can consider any step in the move sequence to be the first one applied to the intermediate
water profile achieved so far, this establishes the claim for finite optimal move sequences.

As any finite move sequence can be simplified in this way without worsening its outcome, the argument applies
to the elements of a sequence of finite move sequences $\Phi=\{\phi^{(m)},\;m\in\N\}$ and thus to an optimal
meta-sequence as well.
\end{proof}

\vspace{1em}
It is tempting to assume that in the case when no finite optimal move sequence exists, we could get away with an
infinite move sequence instead of a sequence of finite move sequences $\Phi$ as described above. However, this is not the case, as the following example shows.

\begin{example}\label{seqofseq}
	\par\begingroup \rightskip15em\noindent Consider the path on four vertices, the target
	vertex not to be one of the end vertices and initial water levels as depicted to the right.
	\par\endgroup
	
	\vspace*{-1.7cm}
	\begin{figure}[H]
		\hspace{7.6cm} \includegraphics[scale=0.55]{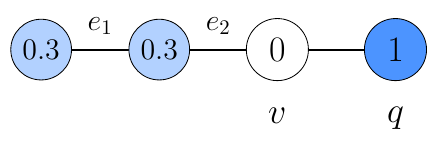}
	\end{figure}
	\vspace*{-0.45cm}
	
	As we will show in Example \ref{linegraph3}, the optimal SAD-profile will allocate $\tfrac16$ of the
	shared glass of water to each of the vertices to the left of $v$ and $v$ itself, the amount of
	$\tfrac12$ to the rightmost vertex $q$, showing that $\kappa(v)=0.6$: First, recall that any SAD-profile
	on a path is unimodal. If $q$ is not the (only)	mode, the contribution of $v$ to the terminal value at $v$ is at least as large as the one from $q$ and thus the SAD-profile in question can not yield
	a water level at $v$ of more than $0.5$, cf.\ \eqref{convcomb}. If $q$ is a mode, the SAD-profile is non-decreasing from
	left to right and thus a flat profile on the vertices other than $q$ uniquely optimal. Finally, to achieve the optimum, the
	contribution of $q$ has to be maximal, i.e.\ $\tfrac12$ (cf.\ Lemma \ref{collection}).
	
	From the considerations in Thm.\ 2.3 in \cite{ShareDrink} it is clear that this SAD-profile, more precisely the value $\tfrac12$
	at $q$, can only be established if the first move is $v$ sharing the drink with $q$ (which corresponds to the last move in the
	water transport -- see Lemma \ref{dual}). Once $v$ starts to share the drink to the left, any other interaction with $q$ will
	decrease the contribution of the latter and thus put a water level of $0.6$ at $v$ out of reach. 
	
	To get a flat profile on three vertices, we need however infinitely many single-edge moves (here on $e_1$ and $e_2$). An optimal meta-sequence of moves is for example given by
	\begin{eqnarray*}
		\Phi&=&\{\phi^{(m)},\;m\in\N\},\quad \text{where}\ \ \phi^{(m)}\in (E\times[0,\tfrac12])^{t_m},\ t_m=2m+1\ \ \text{and}\\
		\phi^{(m)}&=&\big(\underbrace{(e_2,\tfrac12),(e_1,\tfrac12),(e_2,\tfrac12), (e_1,\tfrac12),\dots}_{m \text{ pairs}},(\langle v,q\rangle,\tfrac12)\big),
	\end{eqnarray*}
	achieving $\eta_{t_m}(v)=0.6-\frac{0.1}{4^m}\to\kappa(v),$
	a value that can not be approached by any stand-alone (finite or) infinite sequence of moves.
\end{example}

When it comes to the opening and closing of pipes, it is not self-evident how far things change if we allow
pipes to be opened simultaneously. First of all one has to properly extend the model laid down in (\ref{update})
by specifying how the water levels behave when more than two locks are opened at the same time. In order to
keep things simple, let us assume that the pipes are all short enough and of sufficient diameter such that we can
neglect all kinds of flow effects. Moreover, let us take the dynamics to be as crude as can be by assuming that
the water levels of the involved barrels approach their common average in a linear and proportional fashion, which
is made more precise in the following definition.

\begin{definition}\label{macro_def}
Given a graph $G=(V,E)$, let $A\subseteq V$ be a set of at least 3 nodes and $E_A\subseteq E$ a set of edges spanning $A$. A {\em hypermove} on $E_A$ (or simply $A$) will denote the action of opening all pipes that
correspond to edges in $E_A$ in some round $k$ simultaneously and will -- analogously to (\ref{update}) -- change
the water levels for all vertices $u\in A$ to
$$\eta_k(u)=(1-2\mu_k)\,\eta_{k-1}(u)+2\mu_k\, \overline{\eta}_{k-1}(A),\quad \text{where }
\overline{\eta}_{k-1}(A)=\frac{1}{|A|}\,\sum_{w\in A}\eta_{k-1}(w)$$
is the average over the set $A$ after round $k-1$ and $\mu_k\in[0,\tfrac12]$.
\end{definition}

First of all, Lemma \ref{dual} transfers immediately and almost verbatim to move sequences including hypermoves:
In a move sequence with a hypermove on the set $A$ in the first round, we get the water levels
$$\eta_1(u)=\begin{cases}\eta_0(u)& \text{if }u\notin A\\
(1-2\mu_1)\,\eta_0(u)+2\mu_1\,\overline{\eta}_0(A)&\text{if }u\in A.
\end{cases}$$
If $\{\xi_{u,v}(t-1),\;u\in V\}$ and $\{\xi_{u,v}(t),\;u\in V\}$ are such that
$$\eta_t(v)=\sum_{u\in V}\xi_{u,v}(t)\,\eta_0(u)=\sum_{u\in V}\xi_{u,v}(t-1)\,\eta_1(u),$$
we find by comparing the coefficient of $\eta_0(u)$
$$\xi_{u,v}(t)=\begin{cases}\xi_{u,v}(t-1)& \text{if }u\notin A\\
(1-2\mu_1)\,\xi_{u,v}(t-1)+\sum_{w\in A} 2\mu_1\,\frac{\xi_{w,v}(t-1)}{|A|}&\text{if }u\in A,
\end{cases}$$
which is the SAD-profile originating from the very same hypermove applied to $\{\xi_{u,v}(t-1),\;u\in V\}$.
With this tool at hand, we can prove the following extension of Lemma \ref{simplif}:

\begin{lemma}\label{simplif2}
Take the network $G=(V,E)$ to be finite, and fix the target vertex $v$ as well as the initial water profile.
\begin{enumerate}[(a)]
 \item Even if we allow hypermoves, the statement of Lemma \ref{simplif} still holds true, i.e.\ reducing the range
       of $\mu_k$ from $[0,\tfrac12]$ to $\{0,\tfrac12\}$ in each round $k$ does not worsen the outcome of optimal
       move \mbox{(meta-)sequences}.
 \item The supremum $\kappa(v)$ of water levels achievable at a vertex $v$, as characterized in Definition
       \ref{kappa}, stays unchanged if we allow move sequences to include hypermoves.
 \end{enumerate}
\end{lemma} 
 
\begin{proof}
\begin{enumerate}[(a)]
	\item Just as in Lemma \ref{simplif}, we consider a move sequence consisting of finitely many (macro) moves 
	-- say again $t\in\N$ -- and especially the SAD-profile dual to the moves after round 1, denoted by 
	$\{\xi_{u,v}(t-1),\;u\in V\}$. If the first action is a hypermove on the set $A$, let us divide its nodes into
	two subsets according to whether their initial water level is above or below the initial average across $A$:
	$$A_a:=\{u\in A,\;\eta_0(u)\geq\overline{\eta}_0(A)\} \quad\text{and}\quad
    A_b:=\{u\in A,\;\eta_0(u)<\overline{\eta}_0(A)\}.$$
    If $\sum_{u\in A_a}\xi_{u,v}(t-1)\leq\sum_{u\in A_b}\xi_{u,v}(t-1)$, increasing $\mu_1$ to $\tfrac12$ will not
    decrease the final water level achieved at $v$. If instead $\sum_{u\in A_a}\xi_{u,v}(t-1)\geq\sum_{u\in A_b}\xi_{u,v}(t-1)$, the same conclusion holds for erasing the first move (i.e.\ setting $\mu_1=0$).
    
	\item Obviously, allowing for pipes to be opened simultaneously can, if anything, increase the maximal water level
	achievable at $v$. However, any such hypermove can be at least approximated by opening
	pipes one after another. Levelling out the water profile on a set of more than 2 vertices completely will correspond to the limit of infinitely many single pipe moves on the edges between them (in a sensible order).
	
	Let us consider a finite move sequence $\phi$, including hypermoves on the sets $A_1\,\dots,A_l$ (in chronological
	order). From part (a) we know that with regard to the final water level achievable at $v$, we can assume w.l.o.g.\
	that all moves are complete averages (i.e.\ $\mu_k=\tfrac12$ for all $k$). Fix $\epsilon>0$ and let us define a
	finite move sequence $\overline{\phi}$ including no hypermoves in the following way:
	We keep all the rounds in $\phi$ in which pipes are opened individually. For the hypermove on $A_i,\ i\in\{1,\dots,l\},$
	we insert a finite number of rounds in which the pipes of an edge set $E_{A_i}$, connecting $A_i$, are opened in
	repetitive sweeps such that the water level at each vertex $u\in A_i$ is less than $\tfrac{\epsilon}{2^i}$ away
	from the average across $A_i$ after these rounds (which is possible according to Lemma \ref{evenout}).
	
	As opening pipes leads to new water levels being convex combinations of the ones before, the differences of
	individual water levels	caused by replacing the hypermoves add up to a total difference of $\sum_{i=1}^l \tfrac{\epsilon}{2^i}<\epsilon$ in the worst case. Consequently, the final water level achieved at $v$ by $\overline{\phi}$ is at most
	$\epsilon$ less than the one achieved by $\phi$. Since $\epsilon>0$ was arbitrary, this proves the claim.
\end{enumerate}\vspace*{-0.35em}
\end{proof}

	Note, however, that the option of hypermoves can make a difference when it comes to the attainability of $\kappa(v)$, see Example \ref{seqofseq} and Theorem \ref{finitemacro} below.

\begin{remark}
	\begin{enumerate}[(a)]
		\item Lemma \ref{simplif2}\,(a) states that even for hypermoves, there is nothing to be gained by closing the pipes before the water levels have balanced out completely. A hypermove on the edge set $E_A$ with $\mu_k=\tfrac12$ can be seen as the limit of infinitely many single-edge moves on $E_A$ in the sense of Lemma \ref{evenout} -- a connection that does not exist for hypermoves with $\mu_k\in(0,\tfrac12)$. In fact, it is not hard to come up with an initial waterprofile on a path consisting of three nodes, where an incomplete hypermove, i.e.\ with $\mu_k\in(0,\tfrac12)$, can not be achieved or even approximated by single-edge moves.
		\item Due to Lemmas \ref{simplif} and \ref{simplif2} we can assume w.l.o.g.\ that the parameters $\mu_k$ in optimal move \mbox{(meta-)sequences} are always equal to $\tfrac12$ in each round, hence omit them and consider a move sequence to be a list of edges only (i.e.\ $\phi\in E^t$).
		We can incorporate a move sequence in which more than one pipe is opened at a time into Definition \ref{opt} by either allowing $\phi_k$, for $k\in\{0,\dots,t\}$, to be a subset of $E$ with more than one element on which the leveling takes place, or by viewing $\phi$ as a limiting case of move sequences $\{\phi^{(m)},\;m\in\N\}$, in which pipes are opened separately, that form a meta-sequence of moves $\Phi$
		-- as just described in the proof of the lemma.
		\end{enumerate}
\end{remark}

\section{Hypermoves turn supremum into maximum}\label{finite}

In order to increase the water level at the target vertex $v$, one could in principle start by
greedily trying to connect the barrels with the highest water levels to the one at $v$.
However, optimizing this strategy is far from trivial. 
If we allow hypermoves, we actually do not need to deal with infinite or meta-sequences of moves as the following
theorem shows.

\begin{theorem}\label{finitemacro}
For a finite graph $G=(V,E)$, an initial water profile $\{\eta_0(u)\}_{u\in V}$ and a fixed
target vertex $v\in V$, the supremum of attainable water levels $\kappa(v)$ can be achieved with a finite number of
hypermoves.
\end{theorem}

\begin{proof}
Given a meta-sequence $\Phi=\{\phi^{(m)},\;m\in\mathbb{N}\}$  consisting of finite move sequences $\phi^{(m)}\in E^{t_m}$, let us define its {\em scope} $\mathcal{S}\subseteq E$ as the set of edges that eventually appear in the move sequences
$\phi^{(m)}$, i.e.\ $$\mathcal{S}(\Phi):=\{e\in E;\ \phi_k^{(m)}=e,\text{ for some }m\in\N,\ 1\leq k\leq t_m\}.$$
Further, for fixed $m\in\mathbb{N}$, let $\sigma_m$ denote the number of {\em sweeps} in $\phi^{(m)}$:
we move along the sequence $\phi^{(m)}$ like a coupon collector and a sweep is completed once every
edge of $\mathcal{S}(\Phi)$ appeared at least once. After each completion we start collecting all over
again.

It is worth to point out two simple facts here: On the one hand, that the scope of a sub-meta-sequence will be a subset of the scope of the original meta-sequence and that $\sigma_m=0$ for all $m\in\mathbb{N}$ is indeed possible; and on the other, that by definition, as long as the thinning of an optimal meta-sequence $\Phi$, in the sense that we remove a part of its elements, leaves us with an infinite meta-sequence $\Phi'\subseteq\Phi$, the latter is itself optimal.

To prove the claim, we are first going to choose an appropriate sub-meta-sequence
$\{\phi^{(m_n)},\;n\in\mathbb{N}\}$ of a given optimal meta-sequence $\Phi=\{\phi^{(m)},\;m\in\mathbb{N}\}$
and then show how it corresponds to a sequence of finitely many hypermoves in the limit
as $n\to\infty$. 

In order to arrive at a sub-meta-sequence that suits our purposes, we will gradually thin out $\Phi$,
giving rise to a finite chain $\Phi\supseteq\Phi_1\supseteq\Phi_2\supseteq...\supseteq\Phi_K$ of
sub-meta-sequences. Let $\N\supseteq \mathcal{N}_1 \supseteq \mathcal{N}_2 \supseteq\dots\supseteq\mathcal{N}_K$
denote the corresponding sets of indices of the finite move sequences retained, i.e.\
$\Phi_i=\{\phi^{(m)},\;m\in\mathcal{N}_i\}$.
We will conduct the thinning with help of a sequence $(L_i)_{i=1}^K$ of finite ordered lists of edges.
Each $L_i$ contains edges (potentially with multiplicity) and is adapted to the current sub-meta-sequence $\Phi_i$ in
such a way that $L_i$ is a subsequence of $\phi^{(m)}$ for all $m\in\mathcal{N}_i$. $L_{i+1}$ arises from $L_i$
by means of insertion of a finite number of edges between two already listed ones. To avoid ambiguities, it will
be specified explicitely which edges of each $\phi^{(m)}$ in $\Phi_i$ the edges listed in $L_i$
correspond to, or rather point at.

This is important for the following notion to be well-defined:
For two consecutive pointers $l_j$ and $l_{j+1}$ in $L_i$ and an element $\phi^{(m)}$ of $\Phi_i$, let
$I^{(m)}(l_j,l_{j+1})$ denote the sequence
of edges in $\phi^{(m)}$ that appear in between the edges corresponding to $l_j$ and $l_{j+1}$
respectively. For ease of notation, we add two symbols to the lists $L_i$ that do not correspond to an
edge: A first element $\ast$ and a last element $\dagger$. In other words, if $L_i$ comprises
$N$ edges, we fix the first and last element in $L_i$ to be $l_0=\ast$ and $l_{N+1}=\dagger$
respectively, and write $I^{(m)}(\ast,l_1)$ for the sequence of moves in $\phi^{(m)}$ before $l_1$ and
$I^{(m)}(l_N,\dagger)$ for the ones after $l_N$, see Figure \ref{pointers} for an illustration.

\begin{figure}[ht]
	\centering
	\includegraphics[scale=0.9]{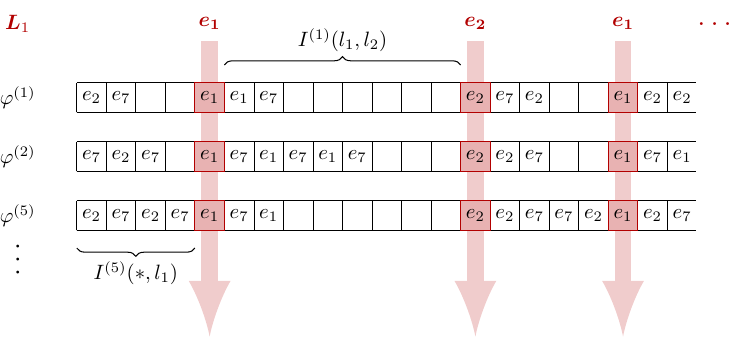}
	\caption{An illustration of the pointers from a list, say $L_1=(e_1,e_2,e_1,\dots)$ to the elements of a meta-sequence $\Phi_1$, with for instance $\mathcal{S}(\Phi_1)=\{e_1,e_2,e_7\}$, $\mathcal{N}_1=\{1,2,5,\dots\}$ and $\phi^{(1)}=(e_2,e_7,e_1,e_1,e_7, e_2, e_7, e_2, e_1, e_2,e_2,\dots)$ etc.\label{pointers}}
\end{figure}
In the process of gradually thinning $\Phi$, we follow an algorithmic divide-and-conquer
approach by passing from $\{\phi^{(m)},\;m\in\mathbb{N}\}$ over to the meta-sequences
$\{I^{(m)}(l_j,l_{j+1}),\;m\in\mathbb{N}\}$ which will be subdivided further whenever the list is extended.
Let us now describe an iteration step of the thinning: We start with
$L_0=(l_0,l_1):=(\ast,\dagger)$ and the meta-sequence
\begin{equation}
\Phi=\{\phi^{(m)},\;m\in\mathbb{N}\} = \{I^{(m)}(\ast,\dagger),\;m\in\mathbb{N}\}.
\end{equation}

If there exists a sub-meta-sequence on which the number of sweeps, i.e.\ $(\sigma_{m_n})_{n\in\mathbb{N}}$,
is bounded, there exists a further sub-meta-sequence on which this number is constant, having value
$M\in\mathbb{N}_0$ say. If $M=0$, we can choose a sub-meta-sequence $\Phi_1\subseteq\Phi$ with reduced scope,
i.e.\ $\mathcal{S}(\Phi_1)\subsetneq\mathcal{S}(\Phi)$. In the case $M\neq0$, there are only finitely many
possible ordered $M$-tuples of edges marking the completions of the sweeps, since $E$ is finite. Hence, we
can choose a sub-meta-sequence $\Phi_1=\{\phi^{(m)},\;m\in\mathcal{N}_1\}$ for which this $M$-tuple is identical and add these $M$ edges to our list
(in the same order and in between $\ast$ and $\dagger$) to get $L_1$, as illustrated in Figure \ref{pointers}. The chosen sub-meta-sequence $\Phi_1$
is divided according to $L_1$ into $M+1$ meta-sequences
$\{I^{(m)}(l_j,l_{j+1}),\;m\in\mathcal{N}_1\}$, $0\leq j\leq M$.
Note that the latter have a scope of cardinality at most $|\mathcal{S}(\Phi)|-1$, just
as with $\Phi_1$ chosen in the case $M=0$.

If $(\sigma_{m})_{m\in\mathbb{N}}$ has no bounded subsequence,
the algorithm terminates without making any changes to the list $L_0$ or the meta-sequence $\Phi$.

After this first iteration, the procedure is successively repeated with respect to some non-empty
$\{I^{(m)}(l_j,l_{j+1}),\;m\in\mathcal{N}_i\}$ in place of $\{I^{(m)}(\ast,\dagger),\;m\in\mathbb{N}\}$:
Say the current meta-sub-sequence is $\Phi_i=\{\phi^{(m)},\;m\in\mathcal{N}_i\}$. Choose $j$ such that
$\{I^{(m)}(l_j,l_{j+1}),\;m\in\mathcal{N}_i\}$ has non-empty scope. If $(\sigma_{m})_{m\in\mathcal{N}_i}$ -- now relating to the sweeps of $\{I^{(m)}(l_j,l_{j+1}),\;m\in\mathcal{N}_i\}$ of its scope
-- has a bounded subsequence, just as above, we can either reduce the scope directly or pick a finite tuple of
edges completing the sweeps, add them to $L_i$ in between $l_j$ and $l_{j+1}$ and apply the chosen thinning of
$\{I^{(m)}(l_j,l_{j+1}),\;m\in\mathcal{N}_i\}$ to $\Phi_i$.
If there is no such subsequence, the algorithm continues with the next pair of consecutive elements in $L_i$
and will never touch the pair $(l_j,l_{j+1})$ again in a future iteration. Due to the strictly decreasing cardinality
of the scope, this algorithmic process will halt after a finite number of iterations and we arrive at a
finite list $$L=(l_0,\dots,l_{N+1})=(\ast,l_1,\dots,l_N,\dagger)$$
as well as a sub-meta-sequence $\{\phi^{(m_n)},\;n\in\mathbb{N}\}$ of
$\Phi$ with the property that the edges $(l_1,\dots,l_N)$ form a subsequence of $\phi^{(m_n)}$ for all
$n$. In addition to that, for every $0\leq j\leq N$, the meta-sequence $\{I^{(m_n)}(l_j,l_{j+1}),\;n\in\N\}$ either has scope $\emptyset$, or the number of times its elements sweep its scope tends to infinity as $n\to \infty$.

The way how to proceed from this point onwards should be rather obvious with Lemma \ref{evenout} in mind.
Let us define the following finite sequence of hypermoves: Keep the single-edge moves $(l_1,\dots,l_N)$
and for $0\leq j\leq N$, put in between $l_j$ and $l_{j+1}$ (in an arbitrary order) hypermoves on the connected components
formed by the scope of $\{I^{(m_n)}(l_j,l_{j+1}),\;n\in\mathbb{N}\}$. By Lemma \ref{evenout}, we can
conclude that the SAD-profile dual to $\phi^{(m_n)}$ converges to the SAD-profile dual to the finite
sequence of hypermoves just described as $n\to\infty$, which concludes the proof.
\end{proof}\vspace{1em}

Theorem \ref{finitemacro} can be seen as a compactness result for the set of achievable water profiles. In contrast
to this constructive, algorithmic approach, the authors in \cite{SAD} independently found a different proof using tools from functional analysis (Thm.\ 4.4).
It is not hard to see that the claim fails for infinite graphs, where even an infinite sequence of hypermoves might not be sufficient to attain the value $\kappa(v)$ at $v$ (cf.\ the
proof of Thm.\ 3.5 in \cite{infinite}). Further, coming back to the case of finite graphs, Theorem \ref{finitemacro}
implies that there always exists an optimal SAD-profile (or limit of SAD-profiles) featuring rational
values only. The fact that any finite number of single-edge moves cannot level out a set, comprising more than two
nodes with different initial water levels, implies that finite optimal move sequences (in the sense of
Definition \ref{opt}) only exist, if there is an optimal meta-sequence for which the algorithm described
in the proof above halts with meta-sequences $\{I^{(m_n)}(l_j,l_{j+1}),\;n\in\mathbb{N}\}$, whose elements
are either all empty or longer and longer strings repeating the same edge (which of course can be contracted to a single move on this edge). In general, however, this is
not the case, as can be seen from the fairly simple Examples \ref{seqofseq} and \ref{linegraph}.

In what follows, we will go back to the initial setting, in which pipes are opened one at a time, but
have in mind (and frequently mention) which hypermove sequence the respective optimal meta-sequence under
consideration corresponds to.

\section{Algorithmic considerations}\label{complexity}

In view of the algorithmic complexity of the problem, it is worthwhile to address the design of
approximative algorithms based on heuristic approaches, even if they might not come arbitrarily close to $\kappa(v)$.

\subsection{Heuristics}\label{optimization}
A somewhat simpler problem, related to the water transport idea, is the concept of {\em greedy lattice
animals} as introduced by Cox, Gandolfi, Griffin and Kesten \cite{GLA1}. They consider the vertices of
a given graph $G$ to be associated with an i.i.d.\ sequence of non-negative random variables and define
a greedy lattice animal of size $n$ to be a connected subset of $n$ vertices containing the target
vertex $v$ and maximizing the sum over the associated $n$ random variables. Since we do not care about
the size of the lattice animal, let us slightly change this definition:

\begin{definition}\label{GLAdef}
	For a fixed (finite) graph $G=(V,E)$, target vertex $v$ and water levels $\{\eta(u)\}_{u\in V}$, let us call
	$C\subseteq V$ a {\em lattice animal (LA)} for $v$ if $C$ is connected and contains $v$. $C$ is a
	{\em greedy lattice animal (GLA)} for $v$ if it maximizes the average of water levels over such sets.
	This average will be considered its value
	$$\mathrm{GLA}(v):=\max_{C\text{ LA for }v}\frac{1}{|C|}\sum_{u\in C}\eta(u).$$
\end{definition}

By Lemma \ref{evenout}, it is clear that $\mathrm{GLA}(v)\leq\kappa(v)$. In fact, for the majority of
settings -- consisting of a graph $G$, a target vertex $v$ and an initial water profile
$\{\eta_0(u)\}_{u\in V}$ -- strict inequality holds and we can do better than just pooling the amount
of water collected in an appropriately chosen connected set of barrels including the one at $v$ (as we
already have seen in Example \ref{seqofseq}).

However, we know from Lemma \ref{simplif2} and Theorem \ref{finitemacro} that the last move of
every finite hypermove sequence achieving water level $\kappa(v)$ will be to pool the amount of water
allocated in a connected set of vertices including $v$. This greedy lattice animal for $v$ (in the
intermediate water profile created up to that point in time) can have a much bigger value than the one
in the initial water profile if we apply the following improving steps first:\\[0.5em]
	\noindent
	\textbf{1) Improving bottlenecks}\\
	Let us call a vertex $u$ a {\em bottleneck} of the GLA $C$ for $v$ if $u\in C\setminus\{v\}$ and
	$\eta(u)<\mathrm{GLA}(v)$.
	Clearly, each bottleneck $u$ has to be a cut vertex for $C$ (otherwise we could just remove it to
	improve the GLA). If there exists a connected subset of vertices $C_u$ including $u$ which has a higher
	average water level than $C_u\cap C$, the value of the GLA for $v$ is improved if the water collected in
	$C_u$ is pooled first (see Figure \ref{GLA1} below). Note that $C_u$ might involve more vertices from $C$ than
	just $u$ (cf.\ Example \ref{linegraph3}).\\[0.5em]
	\noindent
	\textbf{2) Enlargement}\\
	The second option to raise the value of the GLA $C$ for $v$ is to apply the idea above to a vertex $u$ in
	the vertex boundary of $C$ in order for the original GLA to be enlarged to a set of vertices in which
	$u$ is a bottleneck. 

	For this to be beneficial, there has to exist a connected set of vertices $C_u$ in
	$V\setminus C$ including $u$ with the following property:
	The average water level in $C_u$ is smaller than $\mathrm{GLA}(v)$ -- otherwise it would be part of $C$ --
	but is raised above this value after improving the potential bottleneck $u$ using water located in
	$V\setminus C$ (see Figure \ref{GLA1}).
	\begin{figure}[h]
		\centering
		\includegraphics[scale=0.9]{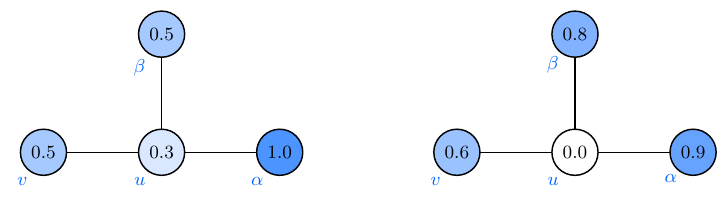}
		\caption{The GLA for target vertex $v$ on the left is $\{v,u,\alpha\}$ with value $\mathrm{GLA}(v)=0.6$, and the bottleneck $u$ can be improved by first opening the pipe $\langle u,\beta\rangle$.\newline     
		The GLA for $v$ with respect to the water profile on the right is $\{v\}$, but can be enlarged
		to $\{v,u,\alpha\}$ if the potential bottleneck $u$ is improved by opening the pipe
		$\langle u,\beta\rangle$ first. \label{GLA1}}
	\end{figure}

	\noindent
	\textbf{3) Choosing the optimal chronological order}\\
	When applying the improving steps just described, it is critical to choose the optimal chronological
	order of moves. Besides the fact that improving bottlenecks and enlarging the GLA has to be done
	before the final averaging, situations can arise in which different sets of vertices can improve the
	same bottleneck or the other way around that more than one bottleneck can be improved using non-disjoint
	sets of vertices, see the illustrative instances depicted in Figure \ref{GLA2}.
	\begin{figure}[H]
		\centering
		\includegraphics[scale=0.9]{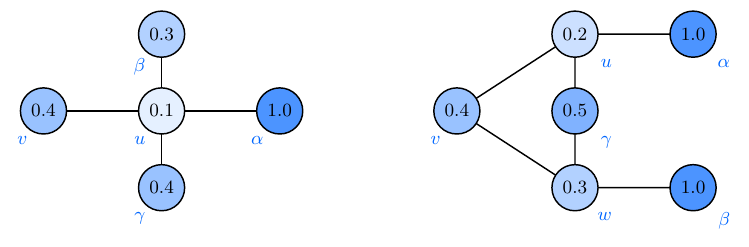}
		\caption{The GLA for target vertex $v$ on the left is $\{v,u,\alpha\}$ with value $\mathrm{GLA}(v)=0.5$.
			Improving the bottleneck $u$ can be done using either $\beta$ or $\gamma$ and is most effective if the pipe $\langle u,\beta\rangle$ is opened first, then $\langle u,\gamma\rangle$.\newline     
			The GLA for $v$ with respect to the graph on the right is $C=\{v,u,w,\alpha,\beta\}$ with value $\mathrm{GLA}(v)=0.58$. The water from $\gamma$ can be used to improve both bottlenecks $u$ and $w$.	
			It is optimal to open pipe $\langle w,\gamma\rangle$ first and then $\langle u,\gamma\rangle$, raising the average water level in $C$ to 0.62.\label{GLA2}}
	\end{figure}

It is worth noticing that lattice animals with lower average than $\text{GLA}(v)$ in the initial
water profile sometimes can be improved by the techniques just described to finally outperform the initial
GLA and all its possible improvements and enlargements (see Example \ref{linegraph3}, especially Figure
\ref{15-line2}).

In fact, for the instance on the right in Figure \ref{GLA2}, it holds $\kappa(v)=\frac{41}{60}\approx 0.683$, achieved by the hypermove sequence $\phi=(\langle w,\gamma\rangle,\langle u,\gamma\rangle,\{u,v,w,\alpha\}, \{v,w,\beta\})$.

\subsection{Heuristics can be far from optimal}
Having this heuristic toolbox at hand, it is possible to devise ad-hoc algorithms that deliver
possibly suboptimal solutions to a given water transport instance (e.g.\ by detecting (near-)optimal
lattice animals and check possible improvements/enlargements). Although both the valid lower bound
of $\text{GLA}(v)$ and common algorithmic designs like divide-and-conquer may at first sight appear to be
promising at least for simple structures, already for networks as simple as trees we cannot hope for
approximation guarantees, as the following two examples show.

\begin{example}\label{starex}
	\begin{figure}[!b]
		\centering
		\includegraphics[scale=0.82]{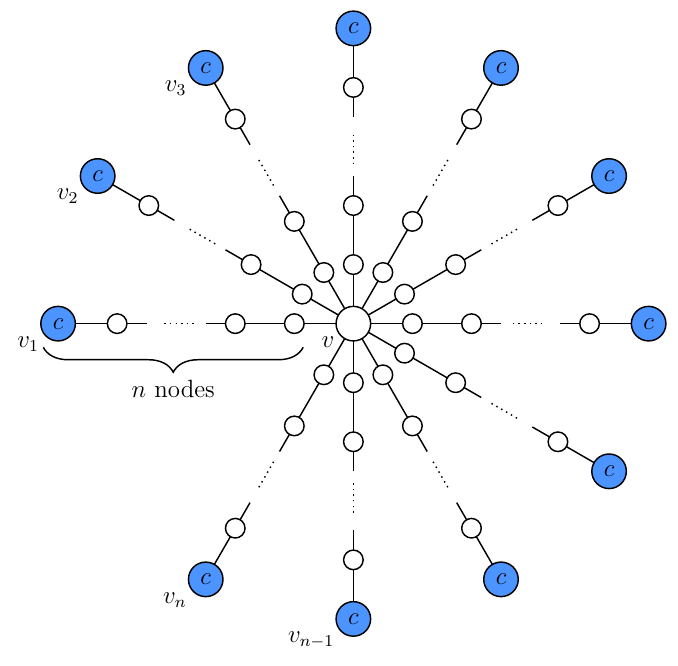}
		\caption{In a star graph, one can get $\kappa(v)\gg\text{GLA}(v)$. \label{star}}
	\end{figure}\noindent
	In order to verify that the value $\text{GLA}(v)$ does not necessarily give a useful approximation on
	$\kappa(v)$, we can consider a symmetric star graph on $n^2+1$ vertices, with $n$ rays and the target
	vertex $v$ at its center. All nodes but the leaves correspond to empty barrels. The leaves are labelled
	$v_1$ through $v_n$ and each of them has the same initial water level $c$, see Figure \ref{star}.

	It is easy to check that, for $c=n+\log(n)+1$ and $n$ large, $\text{GLA}(v)=\tfrac{n\,c}{n^2+1}\approx 1$.
	For this instance, we can however devise a water transport strategy showing $\kappa(v)\geq\log(n)\gg 1$.
	In order to do so, let us define, for $2\leq k\leq n$, $B_k$ to be the minimal connected superset of $\{v_1,\dots,v_k\}$ and $A_k:=B_k\setminus\{v_1,\dots,v_{k-1}\}$, as well as $A_1$ to be the path from $v$ to $v_1$. Note that for all $k$, the set $A_k$ contains exactly one leaf and $|A_k|=k\,(n-1)+2$.
	
	If we follow the strategy to pool the water in the sets $A_n, A_{n-1},\dots,A_1$ in this chronological
	order (for simplicity reasons consider hypermoves), we arrive at the following: As long as the water
	level at $v$, and hence all other vertices in $A_k\setminus\{v_k\}\subseteq A_{k+1}$ is below $\log(n)$,
	the hypermove on $A_k$ will distribute the excess of water from $v_k$ (which is at least $n+1$) and add
	an amount to $v$ bounded from below by $\frac{1}{|A_k|}(c-\log(n))\geq \tfrac{1}{k}$. Hence, if we perform
	all $n$ rounds, the water level at $v$ will be raised to at least $\sum_{k=1}^n\frac1k\geq\log(n+1)$.
\end{example}
\noindent
\begin{minipage}{0.55\textwidth}
	\begin{example}
		The water transport instance depicted to the right reveals that tackling the optimization problem
		by breaking the underlying network into smaller pieces can be quite far from optimal as well. Such a
		divide-and-conquer approach may seem especially tempting on trees, but even there does not
	    work well in general. We consider the tree to the right to be rooted at $v$. Solving
		first the water transport problem for the subtree rooted at $u$, we would pool with the right
		branch first, then $u_1$ and get a water level of roughly $n+\frac12$. Using this water (or even
		the cumulated amount of the pair $\{u,u_1\}$) to raise the level at $v$ will result in something
		of order $\frac1n$, while the strategy consisting of two hypermoves, namely pooling the water
		along the path connecting $v$ to $u_2$ first, then along the one connecting it to $u_1$, shows
		$\kappa(v)\geq\frac{n^5+2n^4+2n^3}{n^5+n^4+n^3+n^2}>1.$
		
		The fact that the global network structure is rather essential in a water transport instance,
		makes the optimization problem in general unamenable to standard implementation schemes, that
		are based on the idea of divide-and-conquer, such as dynamic programming for example.
	\end{example}
	\end{minipage}\hspace*{0.4cm}
	\begin{minipage}{0.4\textwidth}
		\begin{figure}[H]
			\includegraphics[scale=0.9]{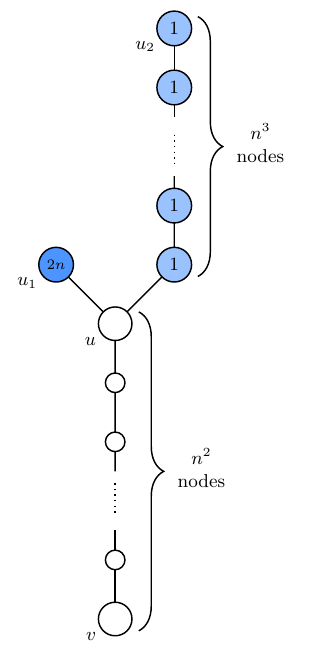}
			\caption{A simple tree not lending itself to a divide-and-conquer approach.}
		\end{figure}
	\end{minipage}

\subsection{Complexity of the problem}\label{complex}

In this subsection, we want to address the complexity of water transport on finite graphs. In fact, we
are going to show that the task of determining whether $\kappa(v)$ is larger than a given
constant -- for a generic instance, consisting of a graph, target vertex and initial water profile -- is
an NP-hard problem. This is done by establishing the following theorem:

\begin{theorem}\label{NPhard}
	The NP-complete problem (exact) 3-SAT can be polynomially reduced to the decision problem of whether
	$\kappa(v)> \lambda$ or not, for a suitably chosen water transport instance and constant $\lambda>0$.
\end{theorem}

Before we deal with the design of an appropriate water transport instance in which to embed the
satisfiability problem {\em exact 3-SAT}, let us provide the definition of Boolean satisfiability problems as
well as known facts about their complexity, for the sake of keeping this part self-contained.

\begin{definition}
	Let $X=\{x_1,x_2,\dots, x_k\}$ denote a set of {\em Boolean variables}, i.e.\ taking on logic truth
	values `TRUE' ($T$) and `FALSE' ($F$). For a variable $x$ in $X$, $x$ and $\overline{x}$ are called
	{\em literals} over $X$. A {\em truth assignment} for $X$ is a function $\tau:X\to\{T, F\}$, where
	$\tau(x)=T$ means that the variable $x$ is set to `TRUE' and $\tau(x)=F$ means that $x$ is set to `FALSE'.
	The literal $x$ is true under $\tau$ if and only if $\tau(x)=T$, its counterpart $\overline{x}$ is true under $\tau$ if and
	only if $\tau(x)=F$.
	
	A {\em clause} $C$ over $X$ is a disjunction of literals and {\em satisfied} by $\tau$ if at least one
	of its literals is true under $\tau$. A logic formula $F$ is in {\em conjunctive normal form (CNF)}, if
	it is the conjunction of (finitely many) clauses. It is called {\em satisfiable} if there exists a
	truth assignment $\tau$ such that all its clauses are satisfied under $\tau$.
	
	The standard Boolean satisfiability problem (often denoted by {\em SAT}) is to decide whether a
	given formula in CNF is satisfiable or not. If we restrict to the case where all the clauses in the
	formula consist of at most 3 literals, it is called {\em 3-SAT}. In {\em exact 3-SAT} each clause has
	to consist of exactly 3 (distinct) literals.
\end{definition}

3-SAT was among the first computational problems shown to be NP-com\-plete, a result published in a
pioneering article by Cook in 1971, see Thm.\ 2 in \cite{3SAT}. It is not hard to see that adding dummy
variables and 7 extra clauses (which enforce the value `FALSE' on the dummy variables) turns any 3-SAT formula into an equisatisfiable exact 3-SAT formula.

Let us now turn to the task of embedding an exact 3-SAT problem into a suitably designed water transport problem, which in size is polynomial in $n$, the number of clauses of the given satisfiability problem:

Given the logic formula $F=C_1\land C_2\land\ldots\land C_n$ in which each of the clauses $C_i$
consists of 3 distinct literals, let us define the comb-like graph depicted in Figure
\ref{reduction}. All the white nodes (including the target vertex $v$) represent empty barrels. The other ones, shaded in blue, contain water to the amount specified.

\begin{figure}[H]
	\hspace{-0.1cm}\includegraphics[scale=0.43]{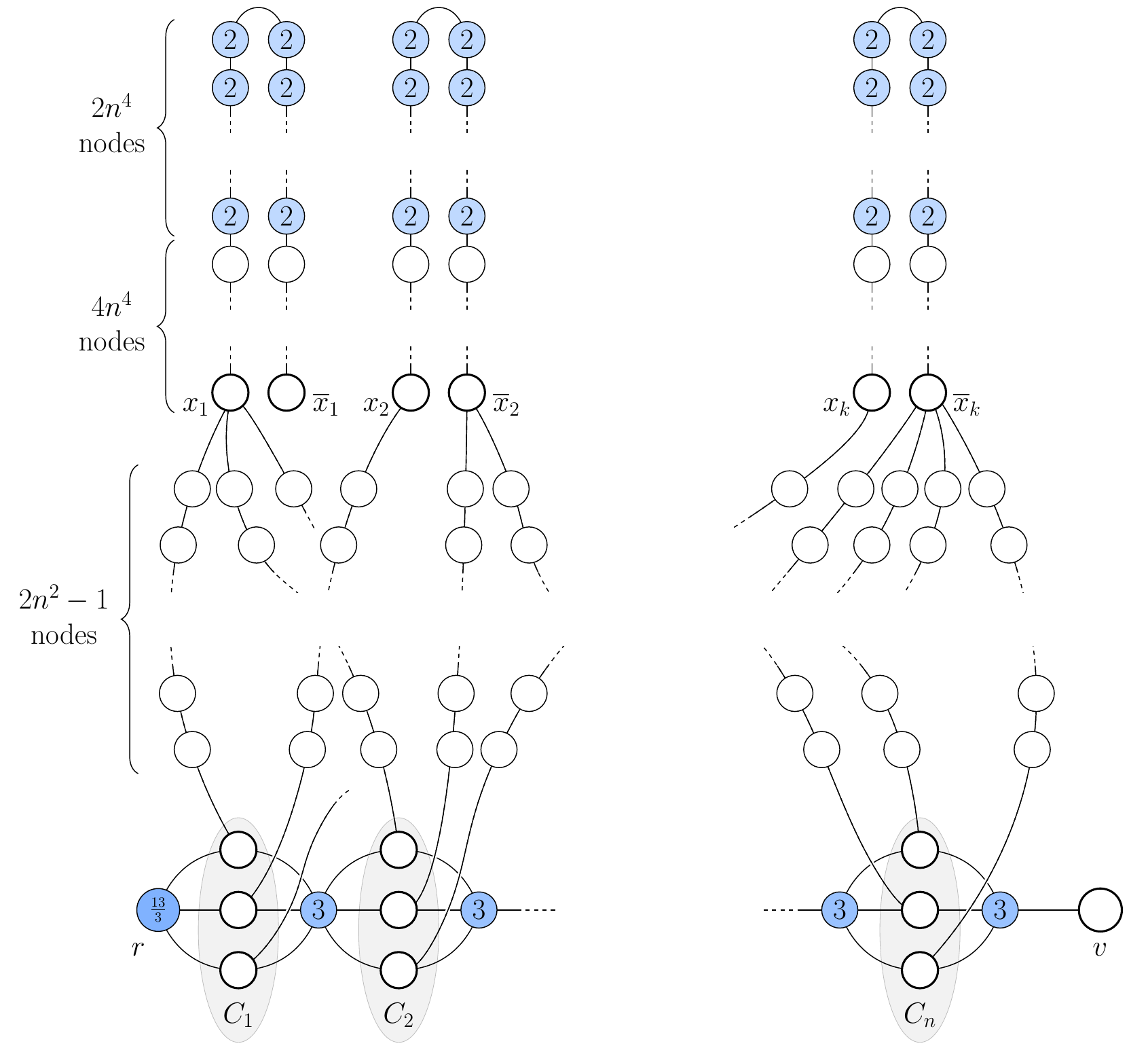}
	\caption{A tailor-made water transport instance on a suitable network graph, comprising $12kn^4+6n^3+n+2$ vertices, to encode a given exact 3-SAT instance.\label{reduction}}
\end{figure}

The comb has $k$ teeth, where $k$ is the number of variables appearing in $F$. Each individual tooth
is formed by a path on $12n^4$ vertices with water level $2$ in each vertex in the middle third of the path.
The lower endvertices of the $i$th tooth are representing the literals $x_i$ and $\overline{x}_i$

The comb's shaft is made up of $4n+2$ vertices, with the target vertex $v$ to the very
right. To the left of $v$ there are $n+1$ vertices representing non-empty barrels: $n$ with water level $3$ as well as vertex $r$ on the opposite end of the shaft with initial water level $\eta_0(r)=\frac{13}{3}$. Any consecutive pair of these vertices are separated by (and in parallel connected by an edge to) three empty barrels representing the literals appearing in the clauses $C_1,\dots,C_n$. These are connected by (disjoint) paths of length $2n^2$ to the matching literal vertex at the bottom end of the teeth.

Observe that the number of nodes in the shaft (clauses) a vertex representing a literal is connected to via a direct path, can vary between $0$ and $n$. For this water transport problem originating from the exact 3-SAT formula $F$ as in
Figure \ref{reduction}, we claim the following:

\begin{proposition}\label{decision}
	Consider the water transport problem based on the logical formula $F$, given by the graph, target
	vertex and initial water profile as depicted in Figure \ref{reduction}.    
	\begin{enumerate}[(a)]
		\item If $F$ is satisfiable, then the water level at $v$ can be raised to a value strictly
		larger than $2$, i.e.\ $\kappa(v)>2$.
		\item If $F$ is not satisfiable, then this is impossible, i.e.\ $\kappa(v)\leq2$.
	\end{enumerate}	
\end{proposition}	

Before we deal with the proof of the proposition, note how it implies the statement of Theorem
\ref{NPhard}: First of all, if $F$ is a 3-SAT formula consisting of $n$ clauses, $k$ cannot exceed
$3\,n$. Given this, it is not hard to check that the graph in Figure \ref{reduction} has $k\cdot 12n^4+3n\cdot2n^2+n+2\leq45n^5$ vertices and maximal degree at most $\max\{6,n+1\}$. As the
initial water levels are all in $\{0,2,3,\tfrac{13}{3}\}$, the size of this water transport instance
is clearly polynomial in $n$. Due to the fact that the value of $\kappa(v)$ can be used to decide
whether the given formula $F$ is satisfiable or not -- as claimed by Proposition \ref{decision} --
Theorem \ref{NPhard} follows.
\vspace*{1em}

\begin{nproof}{of Proposition \ref{decision}}
	\begin{enumerate}[(a)]
	\item 
	To prove the first part of the proposition, let us assume that $F$ is satisfiable. Then there
	exists a truth assignment $t$ with the property that all clauses $C_1,\dots,C_n$ contain at
	least one of the $k$ literals that are set true by $t$. Those can be used to let the water
	trickle down from the teeth to the shaft in an effective way:	
	We assign each clause to one of the true literals under $t$ which it contains. Then, we
	average the water over $k$ (disjoint) star-shaped trees. Each such tree has a literal
	$x\in\{x_1,\overline{x_1},\dots,x_k,\overline{x_k}\}$, which is true under $t$, at its center. Included in the star with center $x$ is the path connecting all nodes with water from the tooth above it, as well as all nodes in the clauses assigned to $x$ (together with the paths of length $2n^2$ connecting them to $x$). If $m$ clauses are assigned to $x$, there are $8\,n^4+2mn^2$ vertices in this star and the accumulated amount of water is $8\,n^4$.
	
	By pooling the water within these $k$ stars, the included nodes in the clauses can simultaneously be pushed
	to a water level as close to the average of the corresponding star as we like (see Lemma \ref{evenout}).
	As $0\leq m\leq n$, these averages are bounded from below by
	$$\frac{8\,n^4}{8\,n^4+2mn^2}=\frac{4\,n^2}{4\,n^2+m}\geq1-\frac{m}{4\,n^2}\geq 1-\frac{1}{4\,n}.$$
	So after this procedure, in each clause there will be one node with water level strictly larger than
	$1-\tfrac{1}{3n}$.
	
	By another complete averaging -- this time over a path on $2n+2$ vertices in the shaft (i.e.\ the nodes
	at the bottom in Figure \ref{reduction}), connecting the vertex with initial water level $\frac{13}{3}$ to $v$ through the clause-vertices assigned to stars in the previous step -- will push the water level at $v$ beyond
	$$\tfrac{1}{2n+2}\,\Big(3n+n\,(1-\tfrac{1}{3n})+\tfrac{13}{3}\Big)=2.$$
	Consequently, for the case of satisfiable $F$, we verified for the graph depicted in Figure
	\ref{reduction} that $\kappa(v)>2$.
	\item To verify the second claim, namely that $\kappa(v)\leq2$ for unsatisfiable $F$, let us start with the observation that no vertex in the paths forming the teeth can achieve a water level higher than 2. Indeed, even if the water from the teeth is used to first fill paths to the shaft to level almost 1 (by an argument similar to Lemma \ref{collection} the resulting profiles are unimodal with the mode in the tooth), the water in the shaft is simply not enough to extend a level of 2 more than a few vertices into the connecting paths (between clause-vertices and literal-vertices).
	
	Let us assume for contradiction that $\kappa(v)>2$. Then there exists a finite sequence of hypermoves $\phi\in (2^V)^t$ realizing a water level exceeding 2 at the target vertex $v$. Let this sequence be chosen both minimal (in the number $t$ of hypermoves) and optimal (maximizing the terminal value $\eta_t(v)$ over all sequences with this number of moves). Then the last move necessarily has to be the (complete) average over a set $C$ including $v$. By choice of $\phi$ all vertices $u\in C\setminus\{v\}$ either have water level $\eta_{t-1}(u)>2$ or are cut-vertices of $C$ (bottlenecks). As a consequence, $C$ can only contain vertices in the shaft: The few vertices in the connecting paths, which potentially could attain a water level higher than 2, would have to be reached via the same way the water exceeding level 1 got there, hence keeping this water in the shaft would in fact decrease the number of steps $t$ necessary -- and potentially also increase $\eta_t(v)$. By a similar argument, at most one out of the three vertices per clause is included in $C$.
	
	Since using water from the teeth only cannot raise the level in a clause-vertex higher than 1 (again unimodularity on paths), for the average accross $C$ at time $t-1$ to be at least 2, vertex $r$ has to be included: If it was not, say $C$ included only $m<n$ clause-vertices, we could potentially use water from $r$ or vertices separating clauses to increase the amount of water in $C$. However, if its clause-barrels are filled through shaft vertices, it would in fact be better to include even these water dispensers into $C$. If they are filled via connecting paths, the contribution of single vertices is bounded from above by $\frac{1}{4n^2+1}$ (part (c) of Lemma \ref{collection}).
	Thus, having $n$ vertices with initial water level $3$, the cummulated water transfered to clause-vertices in $C$ (beyond a water level of 1) is bounded from above by $(2n+\frac{10}{3})\cdot\frac{m}{4n^2+1}$, which is less than 1 (since $m\leq n-1$). This excess water is however not enough as $C$ then includes $2m+1$ vertices and the total amount of water is at most
	$3m+m+1$. By the same logic, accessing target vertex $v$ before the final move actually does not increase the amount of water in $C$.
	
	This leaves one final question to be settled: Can all $n$ clause-vertices in $C$ be pushed to have a water level close enough to 1, when $F$ is not satisfiable and the strategy described in part (a) fails?
	
	If the water from tooth $i$ is used to fill clause-vertices connected to both $x_i$ and $\overline{x}_i$ (not touching the vertices separating clauses), at least one of them cannot exceed water level $\frac23$ (as including the whole path forming tooth $i$ pushes the number of involved vertices beyond $12n^4$, with still only $8n^4$ accumulated entities of water available). Since using the water from a tooth only cannot raise the level in any clause-vertex higher than 1 (as mentioned above), the total amount of water in $C$ is hence bounded from above by $3n+(n-1)+\frac23+\frac{13}{3}=4(n+1)$ which is exactly $2|C|$.
	If the water level in vertices in the shaft is raised by moves including vertices separating clauses, the total amount of water in $C$ will be unchanged or even decrease. Consequently, our initial assumption that we can create a GLA for $v$ with average larger than 2 is false if $F$ is unsatisfiable and selecting one of each pair of literals does not cover all clauses. This concludes the proof.
	\vspace*{-1em}
\end{enumerate}\vspace*{-1em}
\end{nproof}

As set out above, this shows that being able to calculate $\kappa(v)$ -- to the extent whether or not it exceeds 2 -- for the comb-like graph depicted in Figure \ref{reduction} also solves the corresponding exact 3-SAT problem, which the graph was based on. Since (exact) 3-SAT is NP-complete, we hereby
established that any problem in NP can be polynomially reduced to a decision problem minor
to the computation of $\kappa(v)$ in a suitable water transport instance -- showing that
computing $\kappa(v)$ in general is indeed an NP-hard problem.

\subsection{Tractable graphs}\label{tractable}

In this subsection, we are going to present some simple finite graphs (namely paths and complete
graphs) for which the optimization problem of water transport is solvable in polynomial time, irrespectively
of the initial profile and chosen target vertex.

\begin{example}[Path of length 2]\label{K2}
	The minimal graph which is non-trivial with respect to water transport is a single edge, in other words
	the complete graph on two vertices:
	$$G=K_2=(\{1,2\},\{\langle1,2\rangle\}).$$
	From Lemma \ref{collection}, we can conclude
	\begin{equation}\label{edgekappa}
	\kappa(1)=\begin{cases}\eta_0(1)&\text{if }\eta_0(1)\geq\eta_0(2)\\
	\tfrac{\eta_0(1)+\eta_0(2)}{2}&\text{if }\eta_0(1)<\eta_0(2).
	\end{cases}
	\end{equation}
\end{example}

\begin{example}[Path of length 3]\label{linegraph}
	\par\begingroup \rightskip14em\noindent
	The simplest non-transitive graph (i.e.\ having vertices of different kind)
	is the path on three vertices:
	\par\endgroup
	
	\vspace*{-1.5cm}
	\begin{figure}[H]
		\flushright \includegraphics[scale=0.9]{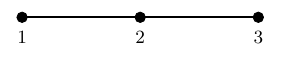}
	\end{figure}
	\vspace*{-0.5cm}
	
	$$G=(\{1,2,3\},\{\langle1,2\rangle,\langle2,3\rangle\}).$$
	Using all three parts of Lemma \ref{collection}, we find the supremum of achievable water levels at
	vertex $1$ to be
	\begin{equation*}
	\kappa(1)=\max\Big\{\eta_0(1),\tfrac{\eta_0(1)+\eta_0(2)}{2},\tfrac{\eta_0(1)+\eta_0(2)+\eta_0(3)}{3}\Big\},
	\end{equation*}
	which is obviously achieved by a properly chosen greedy lattice animal.
	
	Consider the case in which the initial water levels satisfy
	\begin{equation}\label{nofin}
	\eta_0(3)\geq\eta_0(2)\geq\eta_0(1)\quad \text{and}\quad \eta_0(3)>\eta_0(1)
	\end{equation}
	Then $\kappa(1)=\tfrac{\eta_0(1)+\eta_0(2)+\eta_0(3)}{3}$ and there exists no finite optimal move sequence.
	This can be seen from the fact that any single move will preserve the inequalities in (\ref{nofin}) and thus
	we have $\eta_t(1)<\kappa(1)<\eta_t(3)$ for all finite move sequences $\phi\in E^t$.
	
	The maximal value achievable by greedy lattice animals at vertex 2 is
	\begin{equation*} \text{GLA}(2)=\max\Big\{\eta_0(2),\tfrac{\eta_0(1)+\eta_0(2)}{2},\tfrac{\eta_0(2)+\eta_0(3)}{2},
	\tfrac{\eta_0(1)+\eta_0(2)+\eta_0(3)}{3}\Big\}. \end{equation*}
	The fact that we can average across one pipe at a time and choose the order of updates allows us to improve over
	this and gives
	\begin{equation}\label{kappa2} \kappa(2)=\max\Big\{\text{GLA}(2),\tfrac12\,\big(\eta_0(1)+\tfrac{\eta_0(2)+\eta_0(3)}{2}\big),
	\tfrac12\,\big(\eta_0(3)+\tfrac{\eta_0(1)+\eta_0(2)}{2}\big)\Big\}. \end{equation}
	To see this, we can take a closer look on the SAD-profiles that can be created by updates along the two
	edges $\langle1,2\rangle$ and $\langle2,3\rangle$ starting from the initial profile $\delta_2=(0,1,0)$: After one
	update -- depending on the chosen edge -- the profile is given by $(\tfrac12,\tfrac12,0)$ or $(0,\tfrac12,\tfrac12)$.
	After a potential second step, we end up with either $(\tfrac12,\tfrac14,\tfrac14)$ or $(\tfrac14,\tfrac14,\tfrac12)$.
	All of the corresponding convex combinations appear on the right hand side of (\ref{kappa2}). By Lemma \ref{evenout},
	we know that continuing like this will finally result in the limiting profile $(\tfrac13,\tfrac13,\tfrac13)$.
	It is not hard to check that any sequence of two or more updates will lead to an SAD-profile of type either
	$(x,\tfrac{1-x}{2},\tfrac{1-x}{2})$ or $(\tfrac{1-x}{2},\tfrac{1-x}{2},x)$, with $x\in[\tfrac14,\tfrac12]$.
	Hence, it can be written as a convex combination of either $(\tfrac12,\tfrac14,\tfrac14)$ 
	and $(0,\tfrac12,\tfrac12)$ or $(\tfrac14,\tfrac14,\tfrac12)$ and $(\tfrac12,\tfrac12,0)$. Consequently, it cannot
	correspond to a final water level at vertex 2 exceeding the value in (\ref{kappa2}).

	In fact, when maximizing the water level for the middle vertex we can neglect the option of leveling out
	the profile completely, since for any initial water profile there is a {\em finite} optimal move sequence $\phi\in E^t$
	achieving $$\eta_t(2)\geq\tfrac13\,\big(\eta_0(1)+\eta_0(2)+\eta_0(3)\big),$$
	as the next example will show.
\end{example}

\begin{example}[Complete graph]\label{complete}
	Given an initial water profile $\{\eta_0(u)\}_{u\in V}$ and the complete graph $K_n$ as underlying network,
	we get for any $v\in V$:
	$$\kappa(v)=2^{-l+1}\,\eta_0(v)+\sum_{i=1}^{l-1} 2^{-i}\;\eta_0(v_i),$$
	where $V$ is ordered such that $\eta_0(v_1)\geq\eta_0(v_2)\geq\dots\geq\eta_0(v_n)$ with $v=v_l$. Furthermore,
	this optimal value can be achieved by a finite move sequence.
	
	To see this is not hard having Lemmas \ref{dual} and \ref{collection} in mind.
	If $v=v_1$, the highest water level is already in $v$ and the best strategy is to stay away from the pipes.
	For $v\neq v_1$, the contribution of vertex $v_1$ -- i.e.\ the share $\xi_{v_1,v}(t)$ in the convex combination of
	$\{\eta_0(u)\}_{u\in V}$ optimizing $\eta_t(v)$, see (\ref{convcomb}) -- can not be more than $\tfrac12$ by part (c) of Lemma \ref{collection}. However, this can be achieved by opening the pipe
	$\langle v,v_1\rangle$.
	According to the duality between water transport and SAD, this is what we do last. The argument just used
	can be iterated for the remaining share of $\tfrac12$ giving that $v_2$ can contibute at most $\tfrac14$
	(given that $v_1$ contributes most possible) and so on.
	Obviously, involving vertices holding water levels below $\eta_0(v)$ can not be beneficial, as all vertices are
	directly connected, so we do not have intermediate vertices being potential bottlenecks.
	
	The optimal finite move sequence $\phi\in E^t$, where $t=l-1$, is then given by 
	$$\phi_k=\langle v,v_{l-k}\rangle,\ k=1,\dots,l-1$$
	leading to 
	$$\eta_k(v)=2^{-k}\,\eta_0(v)+\sum_{i=1}^{k} 2^{-k+i-1}\;\eta_0(v_{l-i})$$
	and consequently $\eta_t(v)=\eta_{l-1}(v)=\kappa(v)$. Note that the option to open several pipes simultaneously
	is useless on the complete graph. Furthermore, the optimal move sequence $\phi$ only includes edges to which $v$ is incident, so the very same reasoning holds for the center $v$ of a star graph with diameter 2 (in particular a path on 3 vertices) as well.
	
	To determine the optimal achievable value at $v$ we have to sort the $n$ initial water levels first. 
	This can be done using the deterministic sorting algorithm `heapsort' which makes $\text{O}(n\,\log(n))$
	comparisons in the worst case. The calculation of $\kappa(v)$ given the sorted list of
	initial water levels needs at most $n-1$ additions and $n-1$ divisions by $2$.
\end{example}

\begin{example}[Path of length $n$, target vertex $v$ an end node]\label{linegraph2}
	Expanding Example \ref{linegraph}, let us reconsider a finite path -- this time not on three
	but $n$ vertices. Let the vertices be labelled $1$ through $n$ and let vertex $1$ (sitting at one end of the path) be
	the target vertex. Given an initial water profile $\{\eta_0(i)\}_{i=1}^n$, $\kappa(1)$ can be determined by $2n-2$
	arithmetic operations ($n-1$ additions, $n-1$ divisions) as it turns out to be
	\begin{equation}\label{kappa_line}\kappa(1)=\max_{1\leq k\leq n}\frac1k \sum_{i=1}^k \eta_0(i).\end{equation}
	In other words, $\kappa(1)$ equals GLA(1), with respect to the initial water profile (see Definition \ref{GLAdef}).
	
	This easily follows from Lemma \ref{collection}, as any achievable SAD-profile $\{\xi_{u,v}(t)\}_{u=1}^n$
	will be non-increasing in $u$. Hence the water level at $1$ will always be a convex combination of averages
	over its $n$ lattice animals and thus bounded from above by the right hand side of \eqref{kappa_line}.
	This value in turn can be at least approximated by averaging over a greedy lattice animal for vertex $1$ in
	the sense of Lemma \ref{evenout}.

	If we allow hypermoves (opening several pipes simultaneously), it is optimal to do just one move, namely
	to open the pipes $\langle1,2\rangle,\dots,\langle l-1,l\rangle$ simultaneously, where $l\in\{1,\dots,n\}$ is chosen such that $\{1,\dots,l\}$ is a GLA for vertex $1$.
\end{example}

\begin{example}[Path of length $n$, general target vertex]\label{linegraph3}
	Finally, let us consider the path on $n$ vertices, with the target vertex $v$ not (necessarily) sitting at one end.
	\begin{figure}[H]
		\centering
		\includegraphics[scale=0.8]{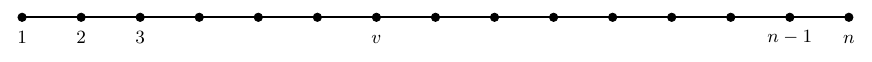}
	\end{figure}
	\noindent
	Given the initial water levels $\{\eta_0(u),\;1\leq u\leq n\}$, let us consider the final SAD-profile
	$\{\xi_{u,v}\}_{1\leq u\leq n}$ corresponding to an optimal move \mbox{(meta-)sequence} (for a meta-sequence,
	it is the limit of its dual SAD-profiles we are talking about, cf. Lemmas \ref{dual} and \ref{simplif}) so that
	\[\kappa(v)=\eta(v)=\sum_{u=1}^n \xi_{u,v}\eta_0(u).\]
	First of all, from Lemma \ref{collection}\,(a) we know that any achievable SAD-profile on a path
	is unimodal (which therefore holds for a pointwise limit of SAD-profiles as well). Let us denote the leftmost
	maximizer of $\{\xi_{u,v}\}_{1\leq u\leq n}$ by $q$ and set
	$$l:=\min\{1\leq u\leq n,\;\xi_{u,v}>0\}\quad\text{and} \quad r:=\max\{1\leq u\leq n,\;\xi_{u,v}>0\}.$$
	Without loss of generality we can assume $l\leq v\leq q\leq r$ (for $q<v$, the set-up simply has to be mirrored), and further that $\{\xi_{u,v}\}_{1\leq u\leq n}$ is chosen to minimize the number $r-l+1$ of involved vertices.
	
%	Furthermore, let us pick the optimal move \mbox{(meta-)sequence} such that $\{\xi_{u,v}\}_{1\leq u\leq n}$ minimizes the distance $d(q,v)=q-v$.
	Since $0<\xi_{l,v}\leq\xi_{l+1,v}\leq\xi_{q-1,v}$ and $\xi_{q,v}\geq\xi_{q+1,v}\geq\xi_{r,v}>0$, the contributions of both the vertices $\{1,\dots,q-1\}$ and $\{q,q+1,\dots,n\}$ respectively can be seen as a scaled-down version of the problem treated in the previous example: This time, the drink to be shared does not amount to 1 but to $\sum_{l\leq u\leq q-1}\xi_{u,v}$ resp.\ $\sum_{q\leq u\leq r}\xi_{u,v}$ instead. From Example \ref{linegraph2} we can therefore conclude that a piecewise flat profile i.e.\ 
	\begin{equation}\label{flatend}
	\xi_{l,v}=\xi_{l+1,v}=\ldots=\xi_{q-1,v}\quad \text{and}\quad\xi_{q,v}=\xi_{q+1,v}=\ldots=\xi_{r,v}
	\end{equation}
	is optimal (as a profile of this form is obviously achievable from $\{\xi_{u,v}\}_{1\leq u\leq n}$), see Figure \ref{SADprof} below for an illustration.
		\begin{figure}[H]
		\centering
		\includegraphics[scale=0.8]{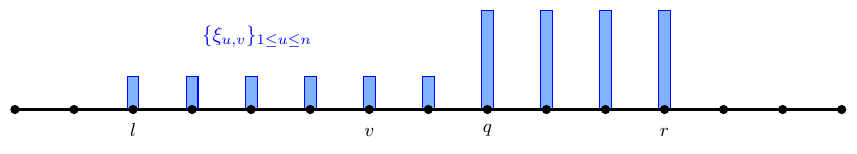}
		\caption{A generic SAD-profile corresponding to an optimal move meta-sequence on a path.\label{SADprof}}
	\end{figure}\noindent
	As $l=v$ would bring us back to Example \ref{linegraph2}, it is enough to consider the case $l<v$, which
	in turn gives $l<q$ by our assumptions. Optimality of $\{\xi_{u,v}\}_{1\leq u\leq n}$ then implies for the corresponding averages
\begin{equation}\label{averages}
	\tfrac{1}{q-l}\sum_{u=l}^{q-1}\eta_0(u)\leq \tfrac{1}{r-q+1}\sum_{u=q}^{r}\eta_0(u),
\end{equation}
	since increasing the contribution of vertices $\{l,\dots,q-1\}$ to the expense of the contribution of vertices $\{q,q+1,\dots,r\}$ would strictly increase the terminal value $\eta(v)$ otherwise. In fact, if equality held in \eqref{averages}, there would by an SAD-profile corresponding to another optimal strategy with a strictly smaller support than $\{\xi_{u,v}\}_{1\leq u\leq n}$, which is a contradiction to our initial choice.

	With strict inequality in \eqref{averages}, optimality dictates that $\xi_{r,v}-\xi_{l,v}$ is maximal.
	By Lemma \ref{collection}\,(c) we know $\xi_{r,v}\leq\tfrac{1}{r-v+1}$ and this value can indeed be achieved in the SAD-process started at $v$ by first sharing equally among $\{v,\dots,r\}$, then equally among $\{l,\dots,q-1\}$.
	By duality (cf.\ Lemma \ref{dual}), this corresponds to two hypermoves in the water transport problem, first on the section $\{l,\dots,q-1\}$ then on $\{v,\dots,r\}$, and leads in this case to
	\[\kappa(v)=\tfrac{q-v}{(q-l)(r-v+1)}\sum_{u=l}^{q-1}\eta_0(u)+\tfrac{1}{r-v+1}\sum_{u=q}^{r}\eta_0(u).\]
	Since $l=v$ together with Lemma \ref{collection}\,(b) implies $q=v$, this formula in fact also applies to that case.
	
	Observe that the values $\xi_{r,v}=\tfrac{1}{r-v+1}$ and $\xi_{l,v}=\tfrac{q-v}{(q-l)(r-v+1)}$ already are determined by the choice of $l,q$ and $r$.
	
	\begin{figure}[H]
		\centering
		\includegraphics[scale=0.4]{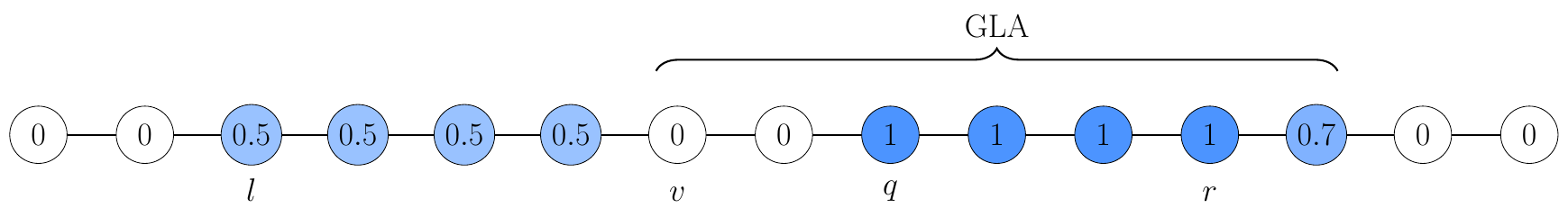}
		\caption{Even for a graph as simple as a finite path, the initial GLA sometimes has little to do with the
			optimal move (meta-)sequence.\label{15-line2}}
	\end{figure}\noindent
	In Figure \ref{15-line2}, a set of initial water levels on the path comprising 15 nodes
	is shown, for which the SAD-profile corresponding to an optimal move meta-sequence is the one depicted in Figure \ref{SADprof} above. From this instance, it can be seen that the GLA with respect to the initial water profile and its possible enhancements can be outperformed by improving another lattice animal as mentioned at the end of Subsection \ref{optimization}.
	
	When it comes to the complexity of finding $\kappa(v)$, we can greedily test all choices for $l,q,r$ (of which
	there are less than $n^3$). For each choice at most $n+3$ additions/subtractions and four multiplications/divisions have
	to be made to calculate either
	\begin{equation}\label{value}\begin{array}{c}
	\tfrac{q-v}{(q-l)\,(r-v+1)}\sum_{u=l}^{q-1}\limits\eta_0(u)+\tfrac{1}{r-v+1}\sum_{u=q}^{r}\limits\eta_0(u) \quad\text{or}\\
	\tfrac{1}{v-l+1}\sum_{u=l}^{\hat{q}}\limits\eta_0(u)+\tfrac{v-\hat{q}}{(r-\hat{q})\,(v-l+1)}\sum_{u=\hat{q}+1}^{r}\limits\eta_0(u),
	\end{array}
	\end{equation}
	depending on whether $v\leq q$ or $\hat{q}\leq v$, where $\hat{q}$ is the rightmost mode of $\{\xi_{u,v}\}_{1\leq u\leq n}$.
	Even though it is possible that there exist SAD-profiles with $q<v<\hat{q}$ corresponding to optimal move \mbox{(meta-)sequences}, by the above we know that
	there has to be one with either $v\leq q$ or $\hat{q}\leq v$ as well.
	The maximal value among those calculated in (\ref{value}) equals $\kappa(v)$, so the complexity is $\mathcal{O}(n^4)$.
	In fact, if we calculate and store all $\binom{n}{2}$ sums over sections of the array of initial water levels
	$\{\eta_0(u)\}_{1\leq u\leq n}$, this running time can be reduced to $\mathcal{O}(n^3)$.
\end{example}

\section{Observations and open problems}\label{concl}

The optimization problem of water transport on finite graphs as presented in this work appears to be quite
elusive from an algorithmic point of view, despite the fact that there always exists an optimal strategy consisting of finitely many (hyper)moves (cf.\ Theorem \ref{finitemacro}).
For instance, there is no monotonicity in the water movement of an optimal move sequence, neither in single pipes nor barrels.\vspace{0.25em}
\begin{minipage}{0.4\textwidth}
	To make this more precise, consider
	the water transport instance depicted on the right. The optimal move sequence consists of two hypermoves here: First use the water of the top barrel to increase the level in the empty barrel, then average over the bottom four vertices to attain $\eta(v)=\kappa(v)=\frac94$.
\end{minipage}\hspace*{0.4cm}
\begin{minipage}{0.55\textwidth}\vspace*{-1.5em}
	\begin{figure}[H]
		\includegraphics[scale=0.9]{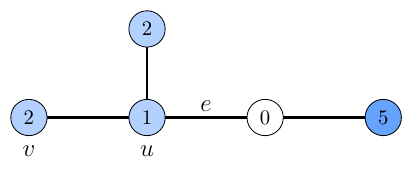}
		\caption{Optimal move sequences do move water back and forth.}
	\end{figure}
\end{minipage}\vspace*{0.12em}

Zooming in on the edge $e$, the water will first move from left to right, then in the opposite direction.
If the initial water levels are changed to $2-\epsilon$ in the top vertex and $1+\epsilon$ at $u$ for some small
$\epsilon>0$, the water level in the barrel at $u$ will first decrease to $1$, then increase to $\frac94$.
Such a mode of action is even possible at the target vertex itself: If in the instance depicted in Figure \ref{15-line2} the vertex just left of $v$ is taken to be the target vertex instead, it is not hard to see that its water level first decreases during the optimal (hyper)move sequence.

In fact, in view of Theorem \ref{finitemacro} it remains an open question, if an optimal hypermove sequence, which is minimal in terms of the number of moves (proven to be finite), can include the same move (i.e.\ averaging over the same set of barrels) more than once. Despite some efforts, no such instance was found and if this was disproved, the minimal number of hypermoves needed to attain water level $\kappa(v)$ at the target vertex would not only be finite, but bounded (by $2^{|V|}$).

When it comes to the
calculation of $\kappa(v)$, the examples in Subsection \ref{optimization} curtail the hope for
a simple approximative algorithm with general approximation guarantees. However, here also lies the potential
for future research: Can one prove any positive results based on approximative algorithms -- if not for general $G$ then at least for a certain class of graphs, such as trees for instance? How far off
can the lower bound based on greedy lattice animals, $\text{GLA}(v)$, be in the worst case (say for barrel capacity and number of vertices fixed)? Example \ref{starex} shows that $\kappa(v)$ can exceed $\text{GLA}(v)$ by a factor as large as $\log(n)$ for a water transport instance of size $\mathcal{O}(n^2)$.

There might be hope to apply approximative algorithms to water transport instances on random networks, e.g.\ the Erd\H{o}s-Rényi graph $G(n,p)$, and analyze the typical gap to $\kappa(v)$, as the most unfavorable graphs would presumably have a rather small chance to appear in such random graph models.

Another interesting aspect is how sensitive the water transport problem is to changes in the network.
While its objective is continuous with respect to the initial water profile, depending on the connectivity of the underlying network graph, it might be a lot more sensitive to the removal of edges (faulty lock) or nodes (leaking barrel).

\subsection*{Acknowledgement}
I'd like to take the opportunity to thank my former supervisor Olle Häggström, who was partly involved in the early stages of this work for both his contributions to it and his valuable feedback at the time.

%\section*{About the author:}
%   We would like a short biographical sketch,
%   beyond just your affiliation to be placed
%   after the bibliography.
%   And below that, your full address.

\vspace{0.5cm}\begin{center}
	\begin{minipage}[t]{0.6\textwidth}
	{\sc \small Timo Vilkas\\
		Statistiska institutionen,\\
		Ekonomihögskolan vid Lunds universitet,\\
		220 07 Lund, Sweden\\}
		timo.vilkas@stat.lu.se\\
	\end{minipage}
\end{center}
\end{document}